\documentclass[12pt,letterpaper]{article}
\usepackage[utf8]{inputenc}
\usepackage{tikz}
\usepackage{tikz-cd}
\usepackage{mathrsfs}
\usepackage{cite}
\usepackage[small,nohug,heads=LaTeX]{diagrams}
\usepackage{amsmath}
\numberwithin{equation}{section}
\usepackage{hyperref}
\usepackage{lineno}
\usepackage{mathrsfs}
\usepackage{amsfonts}
\usepackage{amsthm}
\usepackage{amssymb}
\newtheorem{define}{Definition}[section]
\newtheorem{example}[define]{Example}
\theoremstyle{remark}
\newtheorem{remark}[define]{Remark}
\theoremstyle{plain}
\newtheorem{theo}[define]{Theorem}
\newtheorem{lemma}[define]{Lemma}
\newtheorem{prop}[define]{Proposition}
\newtheorem{cor}[define]{Corollary}
\newcommand{\X}{\mathscr X}
\newcommand{\C}{\mathscr C}
\newcommand{\E}{\mathscr E}

\newcommand{\A}{\mathscr A}

\newcommand{\D}{\mathscr D}

\newcommand{\K}{\mathfrak K}

\newcommand{\Hom}{\mathfrak{Hom}}
\newcommand{\cat}{\mathfrak{Cat}}

\usepackage{amssymb}
\usepackage[left=2cm,right=2cm,top=2cm,bottom=2cm]{geometry}
\author{Michael Lambert}
\title{Computing Weighted Colimits}
\begin{document}
\maketitle

\begin{abstract}
A well-known result of SGA4 shows how to compute the pseudo-colimit of a category-valued pseudo-functor on a 1-category.  The main result of this paper gives a generalization of this computation by constructing the weighted pseudo-colimit of a category-valued pseudo-functor on a 2-category.  From this is derived a computation of the weighed bicolimit of a category-valued pseudo-functor.
\end{abstract}

\tableofcontents

\section{Introduction}

The present paper has its origin in the need for explicit constructions of weighted pseudo-colimits of diagrams valued in certain 2-categories.  The intended application is a characterization of flat pseudo-functors in terms of higher filteredness conditions, generalizing the well-known characterization of ordinary flat set-valued functors (see for example Ch. VII of \cite{MM}).  The question has been studied in the recent papers of Descotte, Dubuc and Szyld \cite{DDS} and \cite{DDS1}.  However, the present approach to the problem will be through the calculus of fractions using the computations of this paper.  More overview on this application appears in \S \ref{prospectus on filteredness}.

The main result of this paper should be seen as a generalization of \S 6.4.0 of \cite{SGA4V2}, which shows how to compute colimits of category-valued pseudo-functors on 1-categories.  The main result, Theorem \ref{MAIN THEOREM}, shows how to compute the pseudo-colimit of any pseudo-functor $E\colon \mathfrak C\to\mathfrak{Cat}$ weighted by $W\colon \mathfrak C^{op}\to\mathfrak{Cat}$ on a small 2-category $\mathfrak C$.  A statement is the following: for any such pseudo-functors $E$ and $W$, a category $E\star W$, obtained as a category of fractions as in \S \ref{colim candidate constru} below, presents the pseudo-colimit of $E$ weighted by $W$ in the following precise sense.

\begin{theo}[Colimit Computation] For pseudo-functors $E\colon \mathfrak C\to\mathfrak{Cat}$ and $W\colon \mathfrak C^{op}\to\mathfrak{Cat}$ on a small 2-category $\mathfrak C$, the category $E\star W$ is, up to isomorphism, the pseudo-colimit of $E$ weighted by $W$ in the sense that there is an isomorphism
\[ \mathfrak{Cat}(E\star W,\X)\cong \Hom(\mathfrak C^{op},\mathfrak{Cat})(W,\mathfrak{Cat}(E,\X))
\]
of categories for any small category $\X$.
\end{theo}

Notice that the colimit universal property in the main theorem formally resembles a tensor-hom adjunction.  This is why the weighted pseudo-colimit is of interest in the characterization of flat pseudo-functors in terms of filteredness conditions.  For it is precisely the exactness of such a tensor product extension that defines flatness.  An explicit computation of the weighted pseudo-colimit will allow an easy treatment of this characterization.

Colimit computations have been of interest historically and recently.  Passing to connected components seems first to have featured in the conical colimit computations of \S I,7.11 of \cite{GrayFormalCats}.  R. Street and E. Dubuc constructed bicolimits of category-valued 2-functors on prefiltered 2-categories in \cite{SD}.  In \S 3.2 of F. Lawler's thesis \cite{LawlerThesis}, there is a computation of conical pseudo-colimits indexed by bicategories similar to the one subsequently presented here.  From this, Lawler computes weighted bicolimits using certain descent diagrams.  The paper \cite{DDS1} of Descotte, Dubuc, and Szyld shows how to compute certain so-called ``$\sigma$-filtered $\sigma$-colimits'' in $\mathfrak{Cat}$.  In the latter two references, weighted colimits are then presented in terms of conical limits.  

In light of the intended application to flatness and filteredness, it seemed worthwhile, however, to give a direct, closed-formed computation of weighted pseudo- and bicolimits from which the conical ones could be otained as a special case.  The insight leading to the present construction in \S \ref{colim candidate constru} is the observation that the category of fractions technique can be carried out for both a category-valued functor and its weight by using a ``diagonal category" $\Delta(E,W)$ that packs the information from the category of elements of each functor into the same category.  Of course 2-cells must be added in to account for the indexing 2-category, but passing to connected components makes the candidate sufficiently 2-categorically discrete not only to be a 1-category but also to satisfy all the necessary 2-dimensional coherence conditions.  

The bulk of the work will be done in \S \ref{assignments and up} where the universal property of the colimit candidate is verified explicitly.  The main result and some corollaries are stated in \S \ref{main theo and corollaries}.  Then \S \ref{bicolims} uses the main theorem to give a construction of weighted bicolimits.  Finally \S \ref{prospectus} discusses further extensions and applications of the main result, including the application to higher filteredness conditions mentioned above.  The rest of this introductory section provides background and notation needed in this and subsequent papers in the form required.

\subsection{2-Categories}

Roughly, the required material on 2-categories is Chapters I,2 and I,3 of Gray's \cite{GrayFormalCats}.  Other references are Chapter 7 of \cite{Handbook1}, Chapter B1 of \cite{Elephant}, and the paper of \cite{KS}.  Some notation and terminology will differ, so here follows a summary of the conventions used throughout. 

A 2-category $\K$ consists of objects, 1-cells, and transformations satisfying well-known axioms (see \S 7.1 of \cite{Handbook1} for example).  Vertical composition of 2-cells will be denoted by juxtapostion `$\beta\alpha$'; while horizontal composition is denoted by `$\ast$' as in $\gamma\ast \delta$.  When horizontally composing a 2-cell with a vertical identity morphism write, for example, `$\alpha\ast f$' or `$g\ast \beta$'.  In general $\K(A,B)$ denotes the vertical category of morphisms $A\to B$ of $\K$ and 2-cells between them.  Each hom category $\K(A,B)$ will be assumed to be small with respect to some fixed category $\mathbf{Set}$.  Any 2-category is a bicategory in the sense of \cite{Benabou} with strict unit and associativity. 

 The notation `$\K^{op}$' indicates the 1-dimensional dual of $\K$; and `$\K^{co}$' denotes the 2-dimensional dual with 2-cells reversed; and `$\K^{coop}$' indicates the 2-category with both 1- and 2-cells reversed.  
 
The basic example is the 2-category $\mathfrak{Cat}$ of small categories relative to a fixed category $\mathbf{Set}$, functors between them, and their natural transformations.  The notation $\mathfrak{CAT}$ will be used for an enlarged 2-category of categories containing a 1-category $\mathbf{Set}$ of sets as an object.  The notation $\mathbf{Cat}$ is used for the 1-category of small categories and functors between them, without considering the 2-dimensional structure.  Generally, any 2-category $\mathfrak A$ has an underlying 1-category $|\mathfrak A|$ obtained by forgetting the 2-cells.  Thus, $|\mathfrak{Cat}|$ and $\mathbf{Cat}$ are essentially notation for the same category.

\begin{define} \label{ps funct defn} A \textbf{pseudo-functor}, or ``homomorphism" as in \cite{Benabou}, between 2-categories $F\colon \mathfrak K\to \mathfrak L$ assigns to each object $A\in \mathfrak K_0$ an object $FA$ of $\mathfrak L$; to each arrow $f$ of $\mathfrak K$ an arrow $Ff$ of $\mathfrak L$; and to each 2-cell $\alpha$ of $\mathfrak K$ a 2-cell $F\alpha$ of $\mathfrak L$; and specifies isomorphisms $\phi_{f,g}\colon FgFf \Rightarrow F(gf)$ and $\phi_A\Rightarrow 1_{FA}\to F1_A$ for each object $A\in\mathfrak K_0$ and composable pair of arrows $f$ and $g$ all satisfying the axioms of Definition B1.1.2 on p.238 of \cite{Elephant}.  Among these is the statement that the equation $F(\beta\alpha)=F\beta F\alpha$ holds for any vertically composable 2-cells $\alpha$ and $\beta$.  There is also the requirement that for horizontally composable 2-cells
$$\begin{tikzpicture}
\node(1){$A$};
\node(2)[node distance=.5in, right of=1]{$\Downarrow \alpha$};
\node(3)[node distance=.5in, right of=2]{$B$};
\node(4)[node distance=.5in, right of=3]{$\Downarrow \beta$};
\node(5)[node distance=.5in, right of=4]{$C$};
\draw[->, bend left](1) to node [above]{$f$}(3);
\draw[->,bend right](1) to node [below]{$g$}(3);
\draw[->,bend left](3) to node [above]{$h$}(5);
\draw[->,bend right](3) to node [below]{$k$}(5);
\end{tikzpicture}$$
the relationship between the images under $F$ is described by the equation
\begin{equation} \label{ps funct coherence cndn} \phi_{g,k}(F\beta\ast F\alpha) = F(\beta\ast \alpha)\phi_{f,h}.
\end{equation}
A pseudo-functor is \textbf{normalized} if the $\phi_A$ as above are identities.  A normalized pseudo-functor is called a \textbf{2-functor} if all of the cells $\phi_{f,g}$ are identities.
\end{define}

\begin{remark} Pseudo-functors will always be assumed to be normalized.
\end{remark}

\begin{define} \label{PSNATTRANSF} A \textbf{pseudo-natural transformation} $\alpha\colon F \Rightarrow G$ between given pseudo-functors $F,G\colon \mathfrak K\rightrightarrows \mathfrak L$ consists of a family of arrows $\alpha_A\colon FA\to GA$ of $\mathfrak L$ indexed over the objects $A\in\mathfrak K_0$ together with, for each arrow $f\in \mathfrak K_1$, an invertible 2-cell
$$\begin{tikzpicture}
\node(1){$FA$};
\node(2)[node distance=1in, right of=1]{$GA$};
\node(3)[node distance=.8in, below of=1]{$FB$};
\node(4)[node distance=.8in, below of=2]{$GB$};
\node(5)[node distance=.5in, right of=1]{$$};
\node(6)[node distance=.4in, below of=5]{$\alpha_f \cong$};
\draw[->](1) to node [above]{$\alpha_A$}(2);
\draw[->](1) to node [left]{$Ff$}(3);
\draw[->](2) to node [right]{$Gf$}(4);
\draw[->](3) to node [below]{$\alpha_B$}(4);
\end{tikzpicture}$$ 
satisfying the following two compatibility conditions:
\begin{enumerate}
\item[]{\textbf{Pseudo-Naturality 1.}} For composable arrows $f$ and $g$ of $\mathfrak K$, there is an equality of 2-cells
$$\begin{tikzpicture}
\node(1){$FA$};
\node(2)[node distance=1in, right of=1]{$GA$};
\node(3)[node distance=.8in, below of=1]{$FB$};
\node(4)[node distance=.8in, below of=2]{$GB$};
\node(5)[node distance=.8in, below of=3]{$FC$};
\node(6)[node distance=.8in, below of=4]{$GC$};
\node(7)[node distance=.5in, right of=1]{$$};
\node(8)[node distance=.4in, below of=7]{$\cong$};
\node(9)[node distance=.8in, below of=8]{$\cong$};
\node(10)[node distance=.7in, right of=4]{$=$};
\node(11)[node distance=2in, right of=8]{$FA$};
\node(12)[node distance=1in, right of=11]{$GA$};
\node(13)[node distance=.8in, below of=11]{$FC$};
\node(14)[node distance=.8in, below of=12]{$GC$};
\node(15)[node distance=1.30in, right of=10]{$\cong$};
\node(16)[node distance=1.1in, right of=15]{$\cong$};
\node(17)[node distance=.4in, right of=16]{$GB.$};
\node(18)[node distance=.4in, left of=3]{$\cong$};
\draw[->](1) to node [above]{$\alpha_A$}(2);
\draw[->](1) to node [left]{$Ff$}(3);
\draw[->](2) to node [right]{$Gf$}(4);
\draw[->](3) to node [below]{$\alpha_B$}(4);
\draw[->](3) to node [left]{$Fg$}(5);
\draw[->](4) to node [right]{$Gg$}(6);
\draw[->](5) to node [below]{$\alpha_C$}(6);
\draw[->](11) to node [above]{$\alpha_A$}(12);
\draw[->](11) to node [left]{$F(gf)$}(13);
\draw[->](12) to node [right]{$G(gf)$}(14);
\draw[->](13) to node [below]{$\alpha_C$}(14);
\draw[->,bend right=80](1) to node [left]{$F(gf)$}(5);
\draw[->,bend left=20](12) to node [above]{$Gf$}(17);
\draw[->,bend left=20](17) to node [below]{$Gg$}(14);
\end{tikzpicture}$$
\item[]{\textbf{Pseudo-Naturality 2.}} For any 2-cell $\theta \colon f\Rightarrow g$ of $\mathfrak K$, there is an equality of 2-cells as depicted in the diagram
$$\begin{tikzpicture}
\node(1){$FA$};
\node(2)[node distance=1.2in, right of=1]{$GA$};
\node(3)[node distance=1in, below of=1]{$FB$};
\node(4)[node distance=1in, below of=2]{$GB$};
\node(5)[node distance=.7in, right of=1]{$$};
\node(6)[node distance=.5in, below of=5]{$\cong$};
\node(7)[node distance=2in, right of=2]{$FA$};
\node(8)[node distance=1.2in, right of=7]{$GA$};
\node(9)[node distance=1in, below of=7]{$FB$};
\node(10)[node distance=1in, below of=8]{$GB.$};
\node(11)[node distance=1.5in, right of=6]{$=$};
\node(12)[node distance=1.5in, right of=11]{$\cong$};
\node(13)[node distance=.4in, below of=1]{$F\theta$};
\node(14)[node distance=.53in, below of=1]{$\cong$};
\node(15)[node distance=.4in, below of=8]{$G\theta$};
\node(16)[node distance=.53in, below of=8]{$\cong$};
\draw[->](1) to node [above]{$\alpha_A$}(2);
\draw[->,bend right](1) to node [left]{$Ff$}(3);
\draw[->,bend left](1) to node [right]{$Fg$}(3);
\draw[->](2) to node [right]{$Gg$}(4);
\draw[->](3) to node [below]{$\alpha_B$}(4);
\draw[->](7) to node [above]{$\alpha_A$}(8);
\draw[->](7) to node [left]{$Ff$}(9);
\draw[->,bend right](8) to node [left]{$Gf$}(10);
\draw[->,bend left](8) to node [right]{$Gg$}(10);
\draw[->](9) to node [below]{$\alpha_B$}(10);
\end{tikzpicture}$$
\end{enumerate}
A pseudo-natural transformation is \textbf{2-natural} if the cells $\alpha_f$ are identities.
\end{define}

\begin{define}[\cite{Benabou}] \label{modification}  A \textbf{modification} $m\colon \theta\Rrightarrow \gamma$ of 2-natural transformations consists of, for each $A\in \mathfrak K_0$, a 2-cell $m_A\colon \theta_A \Rightarrow \gamma_A$ of $\mathfrak L$ satisfying the following compatibility condition:
\begin{enumerate}
\item[]{\textbf{Modification Condition.}} \label{modification condition} There is an equality of 2-cells
$$\begin{tikzpicture}
\node(1){$FA$};
\node(2)[node distance=1.6in, right of=1]{$GA$};
\node(3)[node distance=.8in, below of=1]{$FA'$};
\node(4)[node distance=.8in, below of=2]{$GA'$};
\node(5)[node distance=.8in, right of=1]{$\Uparrow m_A$};
\node(6)[node distance=.9in, below of=5]{$$};
\node(7)[node distance=.5in, below of=5]{$$};
\node(8)[node distance=1.7in, right of=7]{$=$};
\node(9)[node distance=1.75in, right of=2]{$FA$};
\node(10)[node distance=1.6in, right of=9]{$GA$};
\node(11)[node distance=.8in, below of=9]{$FA'$};
\node(12)[node distance=.8in, below of=10]{$GA'$};
\node(13)[node distance=.8in, right of=9]{$$};
\node(14)[node distance=.2in, below of=13]{$$};
\node(15)[node distance=.6in, below of=14]{$\Uparrow m_{A'}$};
\draw[->](2) to node [right]{$Gf$}(4);
\draw[->](1) to node [left]{$Ff$}(3);
\draw[->,bend left](1) to node [above]{$\gamma_A$}(2);
\draw[->,bend right](1) to node [below]{$\theta_A$}(2);
\draw[->,bend right](3) to node [below]{$\theta_{A'}$}(4);
\draw[->,bend left](9) to node [above]{$\gamma_A$}(10);
\draw[->](9) to node [left]{$Ff$}(11);
\draw[->,bend left](11) to node [above]{$\gamma_{A'}$}(12);
\draw[->,bend right](11) to node [below]{$\theta_{A'}$}(12);
\draw[->](10) to node [right]{$Gf$}(12);
\end{tikzpicture}$$
for each arrow $f\colon A\to A'$ of $\mathfrak K$.
\end{enumerate}
\end{define}
Throughout $[\mathfrak K,\mathfrak L]$ denotes the 2-category of 2-functors $\mathfrak K\to\mathfrak L$, their 2-natural transformations, and modifications between them.  In particular, $[\mathfrak K^{op},\mathfrak{Cat}]$ could be considered an appropriate 2-dimensional analogue of the ordinary category of presheaves on a small 1-category.  Pseudo-functors $\K\to\mathfrak{L}$, together with pseudo-natural transformations and modifications between them, form a 2-category, denoted $\Hom(\K,\mathfrak L)$.  Thus, in particular, $\Hom(\K^{op},\mathfrak{Cat})$ denotes the 2-category of category-valued pseudo-functors, pseudo-natural transformations, and modifications, another candidate for 2-dimensional presheaves. For covariant category-valued pseudo-functors use subscripted `$!$' to denote the transition morphisms and 2-cells; and use superscripted `$*$' for the contravariant pseudo-functors as in `$f_!$' and `$f^*$', respectively, for example.

\subsection{Connected Components}

\label{conn comp section}

Every set is a category whose objects and arrows are just the members of the set.  Such a category is ``discrete" and there is a ``discrete category" functor $disc\colon \mathbf{Set}\to\mathbf{Cat}$ of ordinary 1-categories.  The discrete category functor is right adjoint to the ``connected components" functor $\pi_0\colon \mathbf{Cat} \to\mathbf{Set}$ given by sending a category $\C$ to the set of its connected components.  In other words, $\pi_0\C$ is given as a coequalizer 
$$\begin{tikzpicture}
\pgfmathsetmacro{\shift}{0.3ex}
\node(1){$\C_1$};
\node(2)[node distance=.9in, right of=1]{$\C_0$};
\node(3)[node distance=.8in, right of=2]{$\pi_0\C$};
\draw[->](2) to node [above]{$$}(3);
\draw[transform canvas={yshift=0.5ex},->](1) to node [above]{$d_0$}(2);
\draw[transform canvas={yshift=-0.5ex},->](1) to node [below]{$d_1$}(2);
\end{tikzpicture}$$
of the domain and codomain functions coming with the category structure.  As observed, for example, in \S I,2.3 of \cite{GrayFormalCats}, there is a similar situation in dimension 2.  For 1-categories can be viewed as ``locally discrete" 2-categories in the sense that all 2-cells are identities.  This extends to a 2-functor $disc\colon \mathfrak{Cat}\to 2\text-\mathfrak{Cat}$.  Again $disc$ has a left adjoint, a ``connected components" functor, given by taking a 2-category $\mathfrak A$ to the 1-category $\pi_0\mathfrak A$, having the same objects and whose morphisms between say $A,B\in\mathfrak A$ are given by taking the connected components of the hom-category
\[ (\pi_0\mathfrak A)(A,B):=\pi_0\mathfrak A(A,B).
\]
In other words, $\pi_0\mathfrak A$ is given by taking connected components locally.  This construction also makes sense for bicategories.  In particular, it is discussed in \S 7.1 of \cite{Benabou} where it is called the ``Poincar\'e category" of the bicategory.

\subsection{Weighted Colimits}

In ordinary category theory, a colimit is an initial object in the category of cocones on a given diagram.  That is, if $D\colon \mathscr J\to\C$ is a functor on a small category $\mathscr J$, a cocone on $D$ with vertex $X\in\C_0$ is a natural transformation $\lambda\colon D\to \Delta X$; a morphism of cocones $m\colon \lambda\to\nu$ is a map of their vertices $f\colon X\to Y$ such that $f\lambda_i=\nu_i$ holds for each $i\in \mathscr J_0$; and finally, the colimit of the diagram $D$ is the initial object of the category of cocones, if it exists.

Now, let $\mathfrak C$ denote a 2-category.  Let $E\colon \mathfrak C\to\mathfrak{Cat}$ and $W\colon \mathfrak C^{op}\to\mathfrak{Cat}$ denote pseudo-functors.  There is a ``hom" 2-functor 
\[ \mathfrak{Cat}(E,-)\colon \mathfrak{Cat}\longrightarrow [\mathfrak C^{op},\mathfrak{Cat}]
\]
given by sending a small category $\X$ to the pseudo-functor 
\[ \mathfrak{Cat}(E,\X)\colon \mathfrak C^{op}\to\mathfrak{Cat}
\]  
given on objects by taking each $C$ of $\mathfrak C^{op}$ to the 1-category of functors and natural transformation $\mathfrak{Cat}(EC,\X)$.  The 2-functor $\mathfrak{Cat}(E,-)$ could also be viewed as taking its values in $\Hom(\mathfrak C^{op},\mathfrak{Cat})$ since every 2-functor is pseudo.  In general, a separate notation will not be used to indicate this change of target.  A 2-cocone on $E$ weighted by $W$ is a 2-natural transformation $W \to\cat(E,\X)$.  A pseudo-cocone on $E$ weighted by $W$ is a pseudo-natural transformation $W\to \mathfrak{Cat}(E,\X)$.  

Throughout ``pseudo" is taken as the primary notion.  And although the following definitions make sense for arbitrary $\K$, the development will be specialized to $\mathfrak K = \mathfrak{Cat}$.

\begin{define}[Weighted Pseudo-Colimit] \label{pscolim defn}  The \textbf{pseudo-colimit} of $E$ weighted by $W$ is a category $E\star W$ together with a pseudo-cocone $\xi\colon W\to \mathfrak{Cat}(E,E\star W)$ inducing an isomorphism of categories
\begin{equation} \label{ps colim up} \mathfrak{Cat}(E\star W,\X)\cong \Hom(\mathfrak C^{op},\mathfrak{Cat})(W,\mathfrak{Cat}(E,\X))
\end{equation}
for any small category $\X$.  A pseudo-colimit is conical if $W$ has $\mathbf 1$ as its only value.  It is finite if $\mathfrak C$ is a finite 2-category and each $WC$ is finitely-presentable.
\end{define}

The weakest colimit-concept considered here is the bicolimit, which relaxes the isomorphism of the pseudo-colimit to an equivalence of categories.  This is the standard weighted (co)limit concept in \cite{Elephant} whereas the weighted pseudo-colimits are referred to as ``strong" weighted limits.

\begin{define}[Weighted Bi-Colimit] \label{bicolim defn}  The \textbf{bi-colimit} of a pseudo-functor $E$ weighted by $W$ is a category $E\star_{bi} W$ together with a pseudo-cocone $\xi\colon W\to \mathfrak{Cat}(E,E\star W)$ inducing an equivalence of categories
\begin{equation} \label{bicolim up} \mathfrak{Cat}(E\star_{bi} W,\X)\simeq \Hom(\mathfrak C^{op},\mathfrak{Cat})(W,\mathfrak{Cat}(E,\X))
\end{equation}
for any small category $\X$.  A bi-colimit is conical if $W$ has $\mathbf 1$ as its only value.  It is finite if $\mathfrak C$ is a finite 2-category and each $WC$ is finitely-presentable.
\end{define}

In some ways, the weighted pseudo-colimits have an ``enriched flavor."  Bi(co)limits, on the other hand, since they are defined via equivalence of categories, are more properly 2-categorical insofar as equivalence is the appropriate notion of sameness in 2-category theory. Certainly, the last definition is appropriate also in the case that $\mathfrak C$ is a genuine bicategory.

\begin{remark}  In general, pseudo-(co)limits present bi-(co)limits.  However, it is not the case that bi-(co)limits are also pseudo-(co)limits.  A simple example illustrating this fact will be given in \S \ref{section comparing} as a result of the main theorem.
\end{remark}

\section{Weighted Colimits of Category-Valued Pseudo-Functors}

The goal of this section is to give the candidate for the weighted colimit computation.  This is done in \S \ref{colim candidate constru}.  Before that it is worthwhile to review the colimit computation result of \cite{SGA4V2}, since the main result should be seen as its 2-categorical generalization.  The first subsection recalls the basics on categories of fractions since this technique is used not only in the Grothendieck-Verdier result but also throughout the rest of the paper.

\subsection{Categories of Fractions}

The details of the category of fractions construction are given in \S I.1 of \cite{GZ} and \S\S 5.1-5.2 of \cite{Handbook1}.  The main result on the topic is the following.

\begin{theo} \label{cat of frac theorem} For any set of morphisms $\Sigma$ of a (small) category $\mathscr C$ there is a category $\mathscr C[\Sigma^{-1}]$ and a canonical functor 
\[ L_{\Sigma}\colon\mathscr C\to\mathscr C[\Sigma^{-1}]
\]
such that the image under $L$ of any morphism of $\Sigma$ is invertible.  This construction is universal in the sense that composition with $L$ induces an isomorphism of 1-categories
\[ \mathfrak{Cat}(\mathscr C[\Sigma^{-1}],\mathscr D)\cong \mathfrak{Cat}(\mathscr C,\mathscr D)_{\Sigma}
\]
where the subscripted `$\Sigma$' indicates the full subcategory of the functor category whose objects are those functors $\mathscr C\to\mathscr D$ inverting every arrow of $\Sigma$.
\end{theo}

Proofs appear in the references.  But roughly speaking, the category $\mathscr C[\Sigma^{-1}]$ is the free category on $\C$ adjoined with certain formal inverses for the arrows of $\Sigma$, modulo the necessary equations imposing the desired invertibility. Thus, the objects of $\mathscr C[\Sigma^{-1}]$ are those of $\mathscr C$, but the arrows are certain ``formal zig-zags'' modulo equations.  

\begin{remark} Thus, to define a functor $\mathscr C[\Sigma^{-1}]\to\mathscr D$ it suffices to define one $\mathscr C\to\mathscr D$ that inverts all the morphisms of $\Sigma$.  A functor from the category of fractions is then canonically obtained making an appropriate commutative triangle.  That the universal property is further an isomorphism of 1-categories means that any natural transformation $F\Rightarrow G$ between functors $F,G\colon \mathscr C\rightrightarrows \mathscr D$ inverting the arrows of $\Sigma$ uniquely extends to one between the induced functors $\mathscr C[\Sigma^{-1}]\rightrightarrows\mathscr D$ as above.
\end{remark}

\begin{remark}[Calculus of Fractions \S I.2 \cite{GZ}] \label{calculus of fractions comment}  There will be occasion to refer to the so-called ``calculus of fractions," though not in an essential way.  This gives ``filtering conditions" on $\Sigma$ under which the category of fractions has a nicer description.  The objects of $\C[\Sigma^{-1}]$ will be the same, but the morphisms will be equivalence classes of spans whose left legs are in $\Sigma$.  A category of fractions arising in this way is easier to work with since the morphisms are not just formal zig-zags under some identifications coming from a quotient of a free category construction.
\end{remark}

\subsection{Computing Conical Pseudo-Colimits}

The discussion of \S 6.4.0 in \cite{SGA4V2} gives a closed-form expression of the pseudo-colimit of a pseudo-functor $F\colon \C\to\mathfrak{Cat}$ on a 1-category $\C$.  Let us recall the context of the result.

For each category-valued pseudo-functor $E\colon \C\to\cat$, on a 1-category, there is a well-known associated category of elements, or ``Grothendieck semi-direct product," namely a functor $\E\to\C$ that turns out to be a cloven opfibration as in \S B1.3 of \cite{Elephant}.  The source category $\E$ has as objects pairs $(C,X)$ with $C\in \C_0$ and $X\in EC$ and as morphisms $(C,x) \to (D,y)$ those pairs $(f,u)$ of morphisms with $f\colon C\to D$ of $\C$ and $u\colon f_!X\to Y$ of $ED$.   Notice that this generalizes the construction of the discrete opfibration associated to a given set-valued functor on a 1-category as detailed in \S II.6 and \S III.7 of \cite{MacLane}.  

Declare a morphism of $\E$ to be ``cartesian" if the component $u$ is invertible.  Let $\Sigma$ denote the set of such cartesian morphisms.  Consider $\mathbf{Cat}(\E,\D)_{\Sigma}$, the set of functors $\E\to\D$ inverting the arrows of $\Sigma$.  This establishes a functor
\[ \mathbf{Cat}(\E,-)_{\Sigma}\colon \mathbf{Cat}\to\mathbf{Set}.
\]
Another way of viewing the result on the existence of a category of fractions for $\Sigma$ is that this functor is representable.  The representing object is the category 
\[ \E[\Sigma^{-1}]
\]
obtained as the category of fractions by formally inverting the cartesian morphisms of $\E$.  The discussion of \S 6.4.0 in \cite{SGA4V2} shows that the representing object $\E[\Sigma^{-1}]$ is the pseudo-colimit of the original pseudo-functor $E\colon \C\to\cat$.

\begin{remark}  Exhibiting a construction merely as a category of fractions does not give a very useful description of its morphisms.  Accordingly, Proposition 6.4 and Proposition 6.5 of \cite{SGA4V2} give conditions under which the limit $\E[\Sigma^{-1}]$ admits a calculus of fractions as in Remark \ref{calculus of fractions comment}.  The condition is on the base category $\C$, requiring that it be pseudo-filtered in the sense of I.2.7 of \cite{SGA4}.  A higher-dimensional analogue appears in further work on filteredness conditions for flat category-valued pseudo-functors (see for example Chapter 5 of \cite{LambertThesis}).
\end{remark}

\subsection{Candidate for the Weighted Pseudo-Colimit}

\label{colim candidate constru}

Let $E\colon \mathfrak C\to\cat$ and $W\colon \mathfrak C^{op}\to\cat$ denote pseudo-functors on a small 2-category $\mathfrak C$.  Let $\Delta(E,W)$ denote the category with objects triples $(C,X,Y)$ with $C\in\mathfrak C_0$ and $X\in EC$ and $Y\in WC$; and with arrows $(C,X,Y)\to (D,A,B)$ those triples $(f,u,v)$ with $f\colon C\to D$ and $u\colon f_!X\to A$ and $v\colon Y\to f^*B$.  Call a morphism $(f,u,v)$ ``cartesian" if both $u$ and $v$ are invertible.  Composition and identities in $\Delta(E,W)$ are as in the category of elements of category-valued pseudo-functors.  Boost $\Delta(E,W)$ up to a 2-category as follows.  Declare a 2-cell $(f,u,v)\Rightarrow (g,x,y)$ to be one $\alpha\colon f\Rightarrow g$ of $\mathfrak{C}$ for which there are commutative triangles
$$\begin{tikzpicture}
\node(1){$$};
\node(2)[node distance=1in, right of=1]{$U$};
\node(3)[node distance=.4in, above of=1]{$f_!X$};
\node(4)[node distance=.4in, below of=1]{$g_!X$};
\node(5)[node distance=1.5in, right of=2]{$V$};
\node(6)[node distance=3.5in, right of=3]{$f^*Y$};
\node(7)[node distance=3.5in, right of=4]{$g^*Y$};
\draw[->](3) to node [left]{$(\alpha_!)_X$}(4);
\draw[->](3) to node [above]{$u$}(2);
\draw[->](4) to node [below]{$x$}(2);
\draw[->](5) to node [above]{$v$}(6);
\draw[->](5) to node [below]{$y$}(7);
\draw[->](6) to node [right]{$(\alpha^*)_Y$}(7);
\end{tikzpicture}$$ 
in $ED$ and $WC$, repectively.  Now, recall from \S \ref{conn comp section} that there is a connected components functor $\pi_0\colon 2\text-\mathfrak{Cat} \to \mathfrak{Cat}$ taking a 2-category $\mathfrak A$ to its 1-category of connected components, given by taking the $\pi_0$ in the usual sense of each hom-category $\mathfrak A(X,Y)$.  Now, the main result will be that  
\begin{equation} \label{colimit construction} E\star W\simeq \pi_0\Delta(E,W)[\Sigma^{-1}]
\end{equation}
that is, that the weighted pseudo-colimit is presented by the category on the right, given by first taking the connected components of the 2-category $\Delta(E,W)$ and then inverting $\Sigma$, the set of images of cartesian morphisms in the resulting 1-category.  The proof of Theorem \ref{MAIN THEOREM} shows this explicitly.  In the meantime, there will be an occasional slight abuse of notation by using $E\star W$ for the right side of \ref{colimit construction}.  For the computations in the next section, note that there is a canonical map $L\colon \Delta(E,W)\to E\star W$ viewing a morphism $(f,u,v)$ as a span with left leg identity.

\begin{remark}  Recall that for any 2- or pseudo-functor $E\colon \mathfrak C\to\mathfrak{Cat}$, there is a 2-category of elements $\mathfrak E$ equipped with a projection 2-functor $\Pi\colon \mathfrak E\to\mathfrak C$, as detailed in \S 1.4 of \cite{BirdThesis} and in \S I,2.5 of \cite{GrayFormalCats}.  This construction is a generalization of the category of elements for set- and category-valued functors on ordinary 1-categories, as reviewed above.  The 2-category of elements construction adds in certain 2-cells from the base 2-category making commutative triangles like those in the display above.  The triangle on the left is be correct for a covariant pseudo-functor whereas the triangle on the right is correct for a contravariant pseudo-functor.  The point of the ``diagonal category" $\Delta(E,W)$ is that it combines these 2-category of elements constructions along a ``diagonal indexing" by $\mathfrak C$ with the appropriate ``mixed variance" exhibited by the reflected shapes of the triangles above.  Note that this is not the 2-category of elements of a product bifunctor, which would not have this ``mixed variance."  The construction of \ref{colimit construction} thus generalizes Lawler's computation in \cite{LawlerThesis} which would treat only one pseudo-functor without its weight.
\end{remark}

\section{Weighted Colimits in 2-Categories of Pseudo-Functors}

The main result of the section, Theorem \ref{MAIN THEOREM}, shows that the colimit candidate of \S \ref{colim candidate constru} computes the pseudo-colimit of $E$ weighted by $W$.  The proof is a direct verification of the universal property as in Definition \ref{ps colim up}.  The first subsection \S \ref{assignments and up} gives the required functorial assignments and verifies the universal property directly.  This is purely technical work exhibiting how passing to connected components and the category of fractions provides the structure necessary to satisfy the required coherence conditions.  The second subsection \S \ref{main theo and corollaries} states the theorem formally and derives some consequences.  Most of the work of this section appeared in \cite{LambertThesis}.

\subsection{Assignments and Universal Property}

\label{assignments and up}

Throughout let $E\colon \mathfrak C\to\mathfrak{Cat}$ and $W\colon \mathfrak C^{op}\to\mathfrak{Cat}$ denote pseudo-functors on a small 2-category $\mathfrak C$.  Throughout use $E\star W$ to denote the colimit candidate of \S \ref{colim candidate constru}.

Now, begin to define a correspondence between categories
\begin{equation} \label{Phi 0}
\Phi\colon \mathfrak{Cat}(E\star W,\X)\longrightarrow \Hom(\mathfrak C^{op},\mathfrak{Cat})(W,\mathfrak{Cat}(E,\X)).
\end{equation} 
Start with a functor $F\colon E\star W\to \X$.  The image under $\Phi$ should be a pseudo-natural transformation $\Phi(F)$ whose components over $C\in \mathfrak C$ should be functors 
\begin{equation} \label{Phi 1}
\Phi(F)_C\colon WC \to \mathfrak{Cat}(EC,\X).
\end{equation}
To define such $\Phi(F)_C$, fix an object $Y\in WC$.  The image should be a functor $EC\to \X$.  For an object $X\in EC$, declare
\begin{equation}
\Phi(F)_C(Y)(X):= F(C,X,Y).
\end{equation}
And for an arrow $u\colon X\to Z$ of $EC$, the image under $\Phi(F)_C(Y)$ is taken to be the image under $F$ of $(1,u,1)$ viewed as a span in $E\star W$ with left leg identity, i.e., as the image of $(1,u,1_Y)$ of $\Delta(E,W)$ viewed in the category of fractions under the canonical map.  Of course this means that $\Phi(F)_C(Y)\colon EC\to \X$ is a functor since $F$ is one.

Now, finish the assignment of \ref{Phi 1}.  For an arrow $v\colon Y\to Z$ of $WC$, declare $\Phi(F)_C(v)$ to be the natural transformation $\Phi(F)_C(Y)\Rightarrow \Phi(F)_C(Z)$ whose components $\Phi(F)_C(v)_X$ are the images of the morphisms $(1,1_X,v)$ viewed as a span with left leg identity.  Naturality in $X\in EC$ and that $\Phi(F)_C$ is a functor both follow because $F$ is a functor.

Now, the components $\Phi(F)_C$ as in \ref{Phi 1} indexed over $C\in\mathfrak C_0$ comprise a pseudo-natural transformation.  To see this, required are invertible cells
$$\begin{tikzpicture}
\node(1){$WD$};
\node(2)[node distance=1.2in, right of=1]{$\mathfrak{Cat}(ED,\X)$};
\node(3)[node distance=.8in, below of=1]{$WC$};
\node(4)[node distance=.8in, below of=2]{$\mathfrak{Cat}(EC,\X)$};
\node(5)[node distance=.6in, right of=1]{$$};
\node(6)[node distance=.4in, below of=5]{$\cong$};
\draw[->](1) to node [above]{$$}(2);
\draw[->](1) to node [left]{$f^*$}(3);
\draw[->](2) to node [right]{$(f_!)^*$}(4);
\draw[->](3) to node [below]{$$}(4);
\end{tikzpicture}$$
for each $f\colon C\to D$ of $\mathfrak C$.  Such a cell should be a natural isomorphism with components indexed over $Y\in WD$.  For such $Y$, the component of the coherence isomorphism should be a natural isomorphism $\Phi(F)_C(f^*Y)\Rightarrow (f_!)^*\Phi(F)_D(Y)$ of functors $EC\to \X$.  For $X\in EC$, a component will be the image under $F$ of the arrow in $E\star W$ given by the span
$$\begin{tikzpicture}
\node(1){$(C,X,f^*Y)$};
\node(2)[node distance=1.4in, right of=1]{$(C,X,f^*Y)$};
\node(3)[node distance=1.6in, right of=2]{$(D,f_!X,Y).$};
\draw[->](2) to node [above]{$1$}(1);
\draw[->](2) to node [above]{$(f,1,1)$}(3);
\end{tikzpicture}$$ 
Note that the image of the span above upon passing to the category of fractions $E\star W$ is an isomorphism.  That these arrows amount to a natural isomorphism results from the fact that $F$ and the canonical morphisms $L$ are functors.  Now, the proposed components of the purported isomorphism in the square above should be natural in $Y\in WD$.  For $v\colon Y\to Z$ in $WD$, the naturality square commutes because $L$ and $F$ are functors.

\begin{lemma}  The components $\Phi(F)_C$ over $C\in \mathfrak C_0$ as in \ref{Phi 1}, with coherence isos as above, are a pseudo-natural transformation.  Thus, the object assigment for $\Phi$ as in \ref{Phi 0} is well-defined.
\end{lemma}
\begin{proof}  Condition 1 of the pseudo-natural transformation axioms in \ref{PSNATTRANSF} can be seen to hold in the following way.  Let $f\colon B\to C$ and $g\colon C\to D$ denote two arrows of $\mathfrak C$.  The equality of the corresponding 2-cells of the form of the first part of the condition then follows from the commutativity of the figure
$$\begin{tikzpicture}
\node(1){$(B,X,f^*g^*Y)$};
\node(2)[node distance=.7in, above of=1]{$$};
\node(3)[node distance=.3in, right of=2]{$(C,f_!X,g^*Y)$};
\node(4)[node distance=2.6in, right of=3]{$(D,g_!f_!X,Y)$};
\node(5)[node distance=.3in, right of=4]{$$};
\node(6)[node distance=.7in, below of=5]{$(D,(gf)_!X,Y)$};
\node(7)[node distance=.7in, below of=6]{$$};
\node(8)[node distance=1.6in, left of=7]{$(B,X,(gf)^*Y)$};
\draw[->](1) to node [left]{$(f,1,1)$}(3);
\draw[->](3) to node [above]{$(g,1,1)$}(4);
\draw[->](4) to node [right]{$(1,\cong,1)$}(6);
\draw[->](8) to node [below]{$\;\;\;\;\;\;\;\;\;\; (gf,1,1)$}(6);
\draw[->](1) to node [below]{$(1,1,\cong)\;\;\;\;\;\;\;\;\;\;$}(8);
\end{tikzpicture}$$
and the fact that the canonical map $L$ and the given $F$ are functors.

The second pseudo-naturality condition of \ref{PSNATTRANSF} is verified in the following way.  Start with a 2-cell $\alpha\colon f\Rightarrow g$ between arrows $f,g\colon C\rightrightarrows D$ of $\mathfrak C$.  The equality of 2-cells in the condition boils down to the commutativity of the square
$$\begin{tikzpicture}
\node(1){$(C,X,f^*Y)$};
\node(2)[node distance=1.6in, right of=1]{$(C,X,g^*Y)$};
\node(3)[node distance=.8in, below of=1]{$(D,f_!X,Y)$};
\node(4)[node distance=.8in, below of=2]{$(D,g_!X,Y)$};
\node(5)[node distance=.5in, right of=1]{$$};
\node(6)[node distance=.4in, below of=5]{$$};
\draw[->](1) to node [above]{$(1,1,\alpha)$}(2);
\draw[->](1) to node [left]{$(f,1,1)$}(3);
\draw[->](2) to node [right]{$(g,1,1)$}(4);
\draw[->](3) to node [below]{$(1,\alpha,1)$}(4);
\end{tikzpicture}$$
when reduced to connected components and subsequently to the category of fractions $E\star W$.  But this can be seen by exhibiting a path between the composite sides of the square, namely, $(f,\alpha,1)$ and $(g,1, \alpha)$.  The path is a 2-cell of $\Delta(E,W)$ between these two arrows.  Take $\alpha$ itself.  The commutative triangles
$$\begin{tikzpicture}
\node(1){$$};
\node(2)[node distance=1in, right of=1]{$g_!X$};
\node(3)[node distance=.4in, above of=1]{$f_!X$};
\node(4)[node distance=.4in, below of=1]{$g_!X$};
\node(5)[node distance=1.5in, right of=2]{$f^*Y$};
\node(6)[node distance=3.5in, right of=3]{$f^*Y$};
\node(7)[node distance=3.5in, right of=4]{$g^*Y$};
\draw[->](3) to node [left]{$(\alpha_!)_X$}(4);
\draw[->](3) to node [above]{$\;\;(\alpha_!)_X$}(2);
\draw[->](4) to node [below]{$1$}(2);
\draw[->](5) to node [above]{$1$}(6);
\draw[->](5) to node [below]{$(\alpha^*)_Y\;$}(7);
\draw[->](6) to node [right]{$(\alpha^*)_Y$}(7);
\end{tikzpicture}$$
show precisely that $\alpha\colon (f,\alpha,1)\Rightarrow (g,1, \alpha)$ is such a 2-cell, hence a path in the localization $E\star W$, meaning that the two arrows in the commutative square reduce to the same class in the localization.  Thus, the images of these classes under $F$ are equal, proving the condition.  \end{proof}

Now, continue the assignments for \ref{Phi 0}.  In particular, take a natural transformation $\alpha\colon F\Rightarrow G$ for functors $F,G\colon E\star W\rightrightarrows \X$.  The image under $\Phi$ should be a modification $\Phi(\alpha)$ with components
\begin{equation} \label{Phi 2}
\Phi(\alpha)_C\colon \Phi(F)_C\Rightarrow \Phi(G)_C
\end{equation}
indexed over $C\in\mathfrak C_0$.  Each such component should be a natural transformation with components 
\begin{equation}  \label{Phi 3}
\Phi(\alpha)_{C,Y}\colon \Phi(F)_C(Y)\Rightarrow \Phi(G)_C(Y)
\end{equation}
indexed by $Y\in WC$.  Further each such component should be a natural transformation
\begin{equation} \label{Phi 4}
\Phi(\alpha)_{C,Y,X}\colon \Phi(F)_C(Y)(X)\Rightarrow \Phi(G)_C(Y)(X)
\end{equation}
indexed over $X\in EC$.  Unpacking the last condition from the definitions, this means that $\Phi(\alpha)_{C,Y,X}$ ought to be an arrow of $\X$ of the form $F(C,X,Y)\to G(C,X,Y)$.  Thus, make the definition
\begin{equation} \label{Phi 5}
\Phi(\alpha)_{C,Y,X}:= \alpha_{C,X,Y} \colon \Phi(F)_C(Y)(X)\to \Phi(G)_C(Y)(X).
\end{equation}
That the collections indicated by the displays \ref{Phi 3} and \ref{Phi 4} are natural in their proper variables follows from the definition in \ref{Phi 5} by the naturality of $\alpha$.  What remains to check is that the components of \ref{Phi 2} comprise a modification.

\begin{lemma}  The arrow assignment for $\Phi$ with components $\Phi(\alpha)_C$ over $C\in\mathfrak C$ as in \ref{Phi 2}  is a modification.  In particular, the arrow assignment for $\Phi$ of \ref{Phi 0} is well-defined.
\end{lemma}
\begin{proof}  Let $f\colon C\to D$ denote an arrow of $\mathfrak C$.  The modification condition in Definition \ref{modification} requires equality of two composite 2-cells making two sides of a cylindrical figure.  Chasing $Y\in WD$ around each composite reveals that the equality will follow from commutativity of the square
$$\begin{tikzpicture}
\node(1){$F(C,X,f^*Y)$};
\node(2)[node distance=1.8in, right of=1]{$F(C,X,f^*Y)$};
\node(3)[node distance=.8in, below of=1]{$F(D,f_!X,Y)$};
\node(4)[node distance=.8in, below of=2]{$F(D,f_!X,Y)$};
\node(5)[node distance=.5in, right of=1]{$$};
\node(6)[node distance=.4in, below of=5]{$$};
\draw[->](1) to node [above]{$\alpha_{C,X,f^*Y}$}(2);
\draw[->](1) to node [left]{$F(f,1,1)$}(3);
\draw[->](2) to node [right]{$F(f,1,1)$}(4);
\draw[->](3) to node [below]{$\alpha_{D,f_!X,Y}$}(4);
\end{tikzpicture}$$
But this is commutative in $\X$ because it is a naturality square for $\alpha$ at the morphism $(f,1,1)$.  \end{proof}

\begin{lemma}  The assignments giving $\Phi$ of \ref{Phi 0} are functorial.
\end{lemma}
\begin{proof}  This follows by the definition of composition of natural transformations on the one hand and of modifications on the other.  \end{proof}

Now, begin assignments for a reverse correspondence, namely, what will be a functor 
\begin{equation} \label{Psi 0}
\Psi\colon \Hom(\mathfrak C^{op},\mathfrak{Cat})(W,\mathfrak{Cat}(E,\X))\longrightarrow \mathfrak{Cat}(E\star W,\X).
\end{equation}
Start with a pseudo-natural transformation $\theta\colon W\to \mathfrak{Cat}(E,\X)$ of the domain.  The image $\Psi(\theta)$ will be a functor; by Theorem \ref{cat of frac theorem} it can be induced from the underlying category $\pi_0\Delta(E,W)$ of $E\star W$ using the universality of the category of fractions construction.  To this end, define
\begin{equation} \label{Psi 1}\Psi(\theta)\colon \pi_0\Delta(E,W)\longrightarrow \X
\end{equation}
in the following way.  On an object $(C,X,Y)$ of the domain, take
\begin{equation} \label{Psi 2}
\Psi(\theta)(C,X,Y) := \theta_C(Y)(X).
\end{equation}
Now, for an arrow assignment, observe first that since $\theta$ is pseudo-natural, it comes with coherence isomorphisms for each arrow $f\colon C\to D$ of $\mathfrak C$ of the form
$$\begin{tikzpicture}
\node(1){$WD$};
\node(2)[node distance=1.2in, right of=1]{$\mathfrak{Cat}(ED,\X)$};
\node(3)[node distance=.7in, below of=1]{$WC$};
\node(4)[node distance=.7in, below of=2]{$\mathfrak{Cat}(EC,\X)$};
\node(5)[node distance=.5in, right of=1]{$$};
\node(6)[node distance=.35in, below of=5]{$\cong$};
\draw[->](1) to node [above]{$\theta_D$}(2);
\draw[->](1) to node [left]{$f^*$}(3);
\draw[->](2) to node [right]{$f_!$}(4);
\draw[->](3) to node [below]{$\theta_C$}(4);
\end{tikzpicture}$$
Denote such a coherence isomorphism by $\theta_f$.  Thus, for a morphism $(f,u,v)$ of $\Delta(E,W)$ with morphisms $u\colon f_!X\to U$ and $v\colon Y\to f^*V$ of the appropriate fibers, take $\Psi(\theta)(f,u,v)$ to be the composite morphism
$$\begin{tikzpicture}
\node(1){$\theta_C(Y)(X)$};
\node(2)[node distance=1.6in, right of=1]{$\theta_C(f^*V)(X)$};
\node(3)[node distance=1.6in, right of=2]{$\theta_D(V)(f_!X)$};
\node(4)[node distance=1.6in, right of=3]{$\theta_D(V)(U)$};
\draw[->](1) to node [above]{$\theta_C(v)_X$}(2);
\draw[->](2) to node [above]{$\theta_{f,V,X}$}(3);
\draw[->](3) to node [above]{$\theta_D(V)(u)$}(4);
\end{tikzpicture}$$
of $\X$.  It must be shown that this induces a well-defined assigment when passing to connected components.

\begin{lemma}  The arrow assignment immediately above is well-defined on connected components.  Additionally, the induced assignment on $\pi_0\Delta(E,W)$ gives a functor $\Psi(\theta)$ as in \ref{Psi 1}.
\end{lemma}
\begin{proof}  The first statement reduces to the case where $\alpha\colon (f,u,v)\Rightarrow (g,x,y)$ is a 2-cell of $\Delta(E,W)$ between arrows $(C,X,Y)\rightrightarrows (D,U,V)$.  The claim is that the top and bottom sides of the outside of the following figure are equal.  
$$\begin{tikzpicture}
\node(1){$\theta_C(f^*V)(X)$};
\node(2)[node distance=1.6in, right of=1]{$\theta_D(V)(f_!X)$};
\node(3)[node distance=.8in, below of=1]{$\theta_C(g^*V)(X)$};
\node(4)[node distance=.8in, below of=2]{$\theta_D(V)(g_!X)$};
\node(5)[node distance=.8in, right of=1]{$$};
\node(6)[node distance=.4in, below of=5]{$$};q
\node(7)[node distance=2.6in, right of=6]{$\theta_D(V)(U)$};
\node(8)[node distance=2.6in, left of =6]{$\theta_C(Y)(X)$};
\draw[->](1) to node [above]{$\theta_f$}(2);
\draw[->,dashed](1) to node [left]{$\theta_C(\alpha)_X$}(3);
\draw[->,dashed](2) to node [right]{$\theta_D(V)(\alpha)$}(4);
\draw[->](3) to node [below]{$\theta_g$}(4);
\draw[->](8) to node [above]{$\theta_C(v)_X$}(1);
\draw[->](8) to node [below]{$\theta_C(y)_X$}(3);
\draw[->](2) to node [above]{$\theta_D(V)(u)$}(7);
\draw[->](4) to node [below]{$\theta_D(V)(x)$}(7);
\end{tikzpicture}$$
But this is immediate.  For the dashed vertical arrows give a square in the center that commutes by the second coherence condition for $\theta_f$ and $\theta_g$ in \ref{PSNATTRANSF} and the two triangles are the images of the commutative triangles coming with the 2-cell $\alpha$ under $\theta_C$ and under $\theta_D(V)$, respectively.  Thus any two such arrows connected by such a 2-cell $\alpha$ are in the same path class.  Since an arbitrary path is just alternating 2-cells of this form, this special case proves the first claim.

Therefore, the assignments for $\Psi$ induce assignments on $\pi_0\Delta(E,W)$.  That the arrow assignment is functorial also follows.  The unit condition is trivial.  That the assignment respects composition is involved but ultimately straightforward.  One sets up a triangular figure each of whose sides is a three-fold composite of morphisms arising as in the arrow assignment.  The claim is that one side of the triangle is equal to the composite of the other two.  This can be seen by filling in the figure with the various naturality and coherence conditions, a tedious but straightforward task.  \end{proof}

\begin{cor}  The functor $\Psi(\theta)\colon \pi_0\Delta(E,W)\to \X$ inverts the images of cartesian morphisms, hence induces a functor on the category of fractions, also denoted by $\Psi(\theta)\colon E\star W\to\X$.  In particular the object assignment of $\Psi$ above in \ref{Psi 0} is well-defined.
\end{cor}
\begin{proof}  The main claim basically follows from the definition of the arrow assignment for $\Psi$.  For if $(f,u,v)$ is cartesian, then $u$ and $v$ are invertible and so are $\theta_C(v)$ and $\theta_D(V)(u)$.  Of course the components of $\theta_f$ are invertible.  Thus, $\Psi(\theta)(f,u,v)$ for such $(f,u,v)$ is invertible in $\X$.  \end{proof}

For an arrow assignment for $\Psi$, begin with a modification $m\colon \theta \Rrightarrow \gamma$ of two given pseudo-natural transformations $\theta,\gamma\colon W\rightrightarrows\mathfrak{Cat}(E,\X)$.  It suffices to induce the required natural transformation from the underlying category $\Delta(E,W)$.  Take an object $(C,X,Y)$.  The evident definition of the required pseudo-natural $\Psi(m)\colon \Psi(\theta)\Rightarrow \Psi(\gamma)$ is just
\begin{equation} \label{Psi 3}
\Psi(m)_{C,X,Y} := m_{C,Y,X}\colon \theta_C(Y)(X)\to \theta_C(Y)(X)
\end{equation}
that is, the $X$-component of the $Y$-component of the $C$-component of the modification $m$.

\begin{lemma}  The definition of \ref{Psi 3} defines a natural transformation.  Thus, in particular, the arrow assignment of $\Psi$ from \ref{Psi 0} is well-defined.  Additionally, $\Psi$, so defined, is a functor.
\end{lemma}
\begin{proof}  That the required naturality squares commutes is just a result of the modification condition \ref{modification condition} satisfied by $m$.  That $\Psi$ is a functor again follows by the definitions of the assignments and the definitions of composition of modifications and of natural transformations.  \end{proof}

\subsection{Main Theorem and Its Consequences}

\label{main theo and corollaries}

The main result of the paper is now the following.

\begin{theo}[Colimit Computation] \label{MAIN THEOREM} The functors $\Phi$ and $\Psi$ of \ref{Phi 0} and \ref{Psi 0} are mutually inverse.  In particular, for pseudo-functors $E\colon \mathfrak C\to\mathfrak{Cat}$ and $W\colon \mathfrak C^{op}\to\mathfrak{Cat}$, the category $E\star W \simeq \pi_0\Delta(E,W)[\Sigma^{-1}]$ as described in \S \ref{colim candidate constru} is the pseudo-colimit of $E$ weighted by $W$ in the sense that $\Phi$ and $\Psi$ thus provide an isomorphism
\[ \mathfrak{Cat}(\pi_0\Delta(E,W)[\Sigma^{-1}],\X)\cong \Hom(\mathfrak C^{op},\mathfrak{Cat})(W,\mathfrak{Cat}(E,\X))
\]
of categories for any small category $\X$, as described in Definition \ref{pscolim defn}.
\end{theo}
\begin{proof}  That $\Phi$ and $\Psi$ are inverse follows by computation from the definitions given over \S \ref{assignments and up}.  That $E\star W\simeq \pi_0\Delta(E,W)[\Sigma^{-1}]$ is the pseudo-colimit follows by Definition \ref{pscolim defn}.  \end{proof}

Here follow some consequences of the main result.  Notice that for $E$ and $\mathfrak C$ as in the statement, the pseudo-colimit construction extends to a 2-functor $E\star - \colon \Hom(\mathfrak C^{op},\mathfrak{Cat})\to\mathfrak{Cat}$.  The assignments on arrows and on 2-cells are the ones suggested by the construction of $E\star W$.

\begin{cor} \label{tensor-hom adjunction} The induced 2-functor $E\star-$ is left 2-adjoint to the 2-functor $\mathfrak{Cat}(E,-)$.
\end{cor}
\begin{proof}  The result of Theorem \ref{MAIN THEOREM} almost proves this.  The isomorphism in the conclusion of the statement is also pseudo-natural in $\X$ and in $W$, as can be seen from the definitions of the morphisms giving the isomorphism.  \end{proof}

Now, if $C\in\mathfrak C_0$ is an object, then consider the colimit weighted by the canonical representable 2-functor $\mathbf yC\colon \mathfrak C^{op}\to\mathfrak{Cat}$.  The computation underlying Theorem \ref{MAIN THEOREM} shows explicitly that $E\star\mathbf yC$ is equivalent to the category $EC$.  For indeed on the one hand there is a functor 
\[ F\colon EC \to \pi_0\Delta(E,\mathbf yC)
\] 
given by 
\begin{equation} \label{F assignments}
F(X) = (C,X,1) \qquad F(u) = (1,u,1)
\end{equation} 
where the latter arrow is viewed reduced modulo connected components in the target.  The assignments for $F$ are completed by then passing to the category of fractions.  Denote the composite again by $F$.  This is plainly a functor.  On the other hand, there is a functor $G\colon \Delta(E,\mathbf yC)\to EC$ given in the following way.  On an object $(B,X,f)$ with $f\colon B\to C$, take the image under $G$ to be the image of $X$ under the transition functor $f_!$, namely, 
\begin{equation} \label{G obj assign} G(B,X,f):= f_!X.
\end{equation}
Now, fix a morphism $(B,X,f) \to (D,Y,g)$ given by $(h,u,\theta)$ with $u\colon h_!X\to Y$ in $ED$ and $\theta\colon f\Rightarrow gh$ a 2-cell of $\mathfrak C$.  The image under $G$ is defined to be the composite
$$\begin{tikzpicture}
\node(1){$f_!X$};
\node(2)[node distance=.9in, right of=1]{$g_!h_!X$};
\node(3)[node distance=.8in, right of=2]{$g_!Y$};
\draw[->](1) to node [above]{$(\theta_!)_X$}(2);
\draw[->](2) to node [above]{$g_!u$}(3);
\end{tikzpicture}$$
where $\theta_!$ is the image under $E$ of the 2-cell $\theta$.  That $G$ is a functor follows by the naturality of the images of the various 2-cells under $E$.  But $\Delta(E,\mathbf yC)$ is also a 2-category.  The assignments for $G$ are well-defined on connected components of $\Delta(E,\mathbf yC)$.  For let $\alpha\colon (h,u,\theta)\Rightarrow (k,v,\gamma)$ denote such a 2-cell.  In particular, the 2-cells $\alpha$, $\gamma$, and $\theta$ satisfy the relationship
$$\begin{tikzpicture}
\node(1){$$};
\node(2)[node distance=1in, right of=1]{$C$};
\node(3)[node distance=.5in, above of=1]{$B$};
\node(4)[node distance=.5in, below of=1]{$D$};
\node(5)[node distance=.33in, right of=1]{$\Downarrow\gamma$};
\node(6)[node distance=3.0in, right of=2]{$C.$};
\node(7)[node distance=2.5in, right of=3]{$B$};
\node(8)[node distance=2.5in, right of=4]{$D$};
\node(9)[node distance=.43in, below of=7]{$\alpha$};
\node(9)[node distance=.57in, below of=7]{$\Leftarrow$};
\node(11)[node distance=.85in, left of=6]{$\Downarrow \theta$};
\node(12)[node distance=.6in, right of=2]{$=$};
\draw[->](3) to node [left]{$k$}(4);
\draw[->](3) to node [above]{$\;\;f$}(2);
\draw[->](4) to node [below]{$g$}(2);
\draw[->](7) to node [above]{$f$}(6);
\draw[->](8) to node [below]{$g$}(6);
\draw[->,bend left](7) to node [right]{$h$}(8);
\draw[->,bend right](7) to node [left]{$k$}(8);
\end{tikzpicture}$$
And so, the images under $G$ of the two 1-cells of $\Delta(E,\mathbf yC)$ above are the left and right sides of the diamond in the following figure:
$$\begin{tikzpicture}
\node(1){$g_!h_!X$};
\node(2)[node distance=.8in, right of=1]{$$};
\node(3)[node distance=.8in, right of=2]{$g_!k_!X.$};
\node(4)[node distance=.6in, above of=2]{$f_!X$};
\node(5)[node distance=.6in, below of=2]{$g_!Y$};
\node(6)[node distance=.25in, above of=2]{$(I)$};
\node(7)[node distance=.25in, below of=2]{$(II)$};
\draw[->](4) to node [above]{$(\theta_!)_X\;\;\;\;\;$}(1);
\draw[->](4) to node [above]{$\;\;\;\;\;\;(\gamma_!)_X$}(3);
\draw[->](1) to node [below]{$g_!u\;\;\;$}(5);
\draw[->](3) to node [below]{$\;\;\;g_!v$}(5);
\draw[->,dashed](1) to node [above]{$$}(3);
\end{tikzpicture}$$
The dashed arrow above is the image under the transition functor $g_!$ of the component $(\alpha_!)_X$.  The triangle (I) commutes by the condition on the 2-cells $\alpha$, $\gamma$, and $\theta$ mentioned above.  The triangle (II) commutes since it is the image under the transition functor $g_!$ of the commutative triangle
$$\begin{tikzpicture}
\node(1){$$};
\node(2)[node distance=1in, right of=1]{$Y$};
\node(3)[node distance=.4in, above of=1]{$h_!X$};
\node(4)[node distance=.4in, below of=1]{$k_!X$};
\draw[->](3) to node [left]{$(\alpha_!)_X$}(4);
\draw[->](3) to node [above]{$\;u$}(2);
\draw[->](4) to node [below]{$\;v$}(2);
\end{tikzpicture}$$
coming by definition with the 2-cell $\alpha$.  In particular, the discussion shows that $G$ extends to a functor on the 1-category of connected components, also denoted by $G\colon \pi_0\Delta(E,\mathbf yC)\to EC$, since every path is constructed from such 2-cells.

\begin{cor} \label{unit for tensor 1} For each $C\in\mathfrak C_0$, the functors $F$ and $G$ in the discussion above induce an equivalence of categories $E\star \mathbf yC \simeq EC$.
\end{cor}
\begin{proof}  In fact, it follows immediately from the definitions that $GF=1$.  On the other hand, it is straightforward, again from the definitions, to construct a natural system of maps $1\Rightarrow FG$, each component of which is a cartesian arrow in $\Delta(E,\mathbf yC)$, hence invertible when passing to the category of fractions, and thus yielding the rest of the equivalence.  \end{proof}

\begin{cor} \label{unit for tensor 2} The equivalence $E\star \mathbf yC \simeq EC$ of Corollary \ref{unit for tensor 1} is pseudo-natural in $C$, yielding a pseudo-natural equivalence $E\star\mathbf y \simeq E$.  In this sense, Yoneda is a unit for the 2-functor $E\star-$.
\end{cor}
\begin{proof}  For an arrow $f\colon C\to D$ of $\mathfrak C$, the required coherence cell
$$\begin{tikzpicture}
\node(1){$EC$};
\node(2)[node distance=1in, right of=1]{$E\star \mathbf yC$};
\node(3)[node distance=.8in, below of=1]{$ED$};
\node(4)[node distance=.8in, below of=2]{$E\star \mathbf yD$};
\node(5)[node distance=.5in, right of=1]{$$};
\node(6)[node distance=.4in, below of=5]{$\phi_f\cong$};
\draw[->](1) to node [above]{$F_C$}(2);
\draw[->](1) to node [left]{$f_!$}(3);
\draw[->](2) to node [right]{$1\star \mathbf yf$}(4);
\draw[->](3) to node [below]{$F_D$}(4);
\end{tikzpicture}$$
has as its $X$-component for $X\in EC$, the arrow $(f,1,1)\colon (C,X,f)\to (D,f_!X,1)$ which is cartesian, hence invertible in $E\star \mathbf yD$.  Naturality in $X$ follows from the definition.  The two pseudo-naturality conditions of Definition \ref{PSNATTRANSF} follow by the construction of the colimit. \end{proof}

\begin{cor}  Every category-valued pseudo-functor $E\colon \mathfrak C\to\mathfrak{Cat}$ is pseudo-naturally equivalent to a strict 2-functor.
\end{cor}
\begin{proof} This is just a restatement of Corollary \ref{unit for tensor 2}.  \end{proof}

\section{Application to Bicolimits}

\label{bicolims}

In \S \ref{bicolim comp as tensor subsection} weighted bicolimits as in Defintion \ref{bicolim defn} are computed as certain pseudo-colimits weighted by a canonical hom-functor.  But first in \S \ref{tensor subsection} the bicolimit candidate is derived using ideas from enriched category theory as a kind of tensor product of pseudo-functors.  The last subsection \S \ref{section comparing} provides some comments on the comparison of pseudo- and bicolimits.

\subsection{The Tensor Product}

\label{tensor subsection}

In the theory of ordinary presheaves, the tensor product of $P\colon \C^{op}\to\mathbf{Set}$ and $Q\colon \C\to\mathbf{Set}$ is a colimit.  But in IX.6 of \cite{MacLane}, it is indicated that the tensor product is also realized as a coend.  That is, for a cocomplete category $\E$ with $Q\colon \C\to\E$, the tensor product is the coend
\[ \int^C(QC)\cdot (PC)
\]
of the copower bifunctor.  Now, in enriched theory, for $\mathscr V$ closed monoidal and $\C$ a $\mathscr V$-category, the copower of $P\in \C_0$ by $Q\in\mathscr V$ is an object $Q\otimes P$ of $\C$ yielding an isomorphism
\[ \C(Q\otimes P,F)\cong \mathscr V(P,\C(Q,F)).
\]
When $\C=\mathscr V$, the copower reduces to the monoidal tensor of $\mathscr V$.  So, in the case that $\mathscr V$ is $\mathbf{Cat}$, a closed symmetric monoidal category under the cartesian product, the copower bifunction is the product bifunctor $Q\times P\colon \C\times\C^{op}\to\mathbf{Set}$.  Enriched theory also tells us how to compute the coend of such a bifunctor $B\colon \C\times \C^{op}\to\mathscr V$.  It is the hom-weighted colimit
\[ \int^CB(C,C):= B\star \C(-,-)
\]
as given in equation (3.66) of \cite{Enriched}.  (Note $\C(-,-)$ is contravariant on $\C\times \C^{op}$.)  

Now, the definition above makes sense even thought the present work is not primarily viewing 2-categories as $\mathbf{Cat}$-enriched.  In this way, it yields a candidate for a ``tensor product" of category-valued pseudo-functors.  That is, define the tensor product of pseudo functors $E\colon \mathfrak C\to\mathfrak{Cat}$ and $W\colon \mathfrak C^{op}\to\mathfrak{Cat}$ to be the coend of the product bifunctor $E\times W\colon  \mathfrak C\times \mathfrak C^{op} \to\mathfrak{Cat}$ as in
\begin{equation} \label{tensor defn} E\otimes_{\mathfrak C}W:= (E\times W)\star \mathfrak C(-,-)
\end{equation}
via the weighted pseudo-colimit constructed in \ref{colim candidate constru}.  Of course the notation is somewhat tendentious.  But the subsequent results provide some further justification.  First note that Theorem \ref{MAIN THEOREM} shows exactly how to compute this tensor product.

\begin{prop} \label{descr bicolim} The tensor product $E\otimes_{\mathfrak C} W$ as in \ref{tensor defn} is, up to equivalence, the category obtained from the 2-category having 
\begin{enumerate}
\item as objects those quintuples
\[ (C,D,X,Y,f) \qquad C,D\in\mathfrak C \qquad X\in EC \qquad Y\in WD\qquad f\colon C\to D\in\mathfrak C_1
\]
\item and as arrows those quintuples
\[(h,k,u,v,\alpha)\colon (C,D,X,Y, f)\to (A, B,U,V, g)
\]
with $\alpha \colon f\Rightarrow kgh$ a 2-cell of $\mathfrak C$ and two vertical arrows $u\colon h_!X\to U$ and $v\colon k^*Y\to V$;
\item and finally as 2-cells $(h,k,u,v,\alpha) \Rightarrow (m,n,x,y,\beta)$ between arrows as above those pairs $(\gamma,\theta)$ of 2-cells $\gamma\colon h\Rightarrow m$, $\theta\colon k\Rightarrow n$ of $\mathfrak C$ for which there are commutative triangles
$$\begin{tikzpicture}
\node(1){$$};
\node(2)[node distance=1in, right of=1]{$U$};
\node(3)[node distance=.4in, above of=1]{$h_!X$};
\node(4)[node distance=.4in, below of=1]{$m_!X$};
\node(5)[node distance=2.5in, right of=2]{$V$};
\node(6)[node distance=2.5in, right of=3]{$k^*Y$};
\node(7)[node distance=2.5in, right of=4]{$n^*Y$};
\draw[->](3) to node [left]{$(\gamma_!)_X$}(4);
\draw[->](3) to node [above]{$u$}(2);
\draw[->](4) to node [below]{$x$}(2);
\draw[->](6) to node [above]{$v$}(5);
\draw[->](7) to node [below]{$y$}(5);
\draw[->](6) to node [left]{$(\alpha^*)_Y$}(7);
\end{tikzpicture}$$ 
and for which there is an equality of composed 2-cells
$$\begin{tikzpicture}
\node(1){$C$};
\node(2)[node distance=1in, right of=1]{$D$};
\node(3)[node distance=.8in, below of=1]{$A$};
\node(4)[node distance=.8in, below of=2]{$B$};
\node(5)[node distance=.5in, right of=1]{$$};
\node(6)[node distance=.4in, below of=5]{$\Downarrow\alpha$};
\node(7)[node distance=1.6in, right of=2]{$C$};
\node(8)[node distance=1in, right of=7]{$D$};
\node(9)[node distance=.8in, below of=7]{$A$};
\node(10)[node distance=.8in, below of=8]{$B$};
\node(11)[node distance=.5in, right of=7]{$$};
\node(12)[node distance=.4in, below of=11]{$\Downarrow\beta$};
\node(13)[node distance=.65in, left of=6]{$\gamma$};
\node(14)[node distance=.12in, below of=13]{$\Leftarrow$};
\node(15)[node distance=.65in, right of=6]{$\theta$};
\node(16)[node distance=.12in, below of=15]{$\Rightarrow$};
\node(17)[node distance=1.4in, right of=6]{$=$};
\draw[->](1) to node [above]{$f$}(2);
\draw[->](1) to node [right]{$h$}(3);
\draw[->,bend right=80](1) to node [left]{$m$}(3);
\draw[->](4) to node [left]{$k$}(2);
\draw[->,bend right=80](4) to node [right]{$n$}(2);
\draw[->](3) to node [below]{$g$}(4);
\draw[->](7) to node [above]{$f$}(8);
\draw[->](7) to node [left]{$m$}(9);
\draw[->](10) to node [right]{$n$}(8);
\draw[->](9) to node [below]{$g$}(10);
\end{tikzpicture}$$
\end{enumerate}
by taking connected components and passing to a category of fractions as in \ref{colimit construction}.
\end{prop}
\begin{proof}  This is just a matter of unraveling the construction of the pseudo-colimit candidate 
\[  E\otimes_{\mathfrak C}W:= (E\times W)\star \C(-,-) \simeq \pi_0\Delta(E\times W,\mathfrak C(-,-))[\Sigma^{-1}]
\]
as in \ref{colimit construction}.  \end{proof}

\begin{remark} Notice that this is much like the ``twisted arrow category'' description of coends as in IX.6 of \cite{MacLane}. 
\end{remark}

\subsection{The Tensor Product Computes Weighted Bicolimits}

\label{bicolim comp as tensor subsection}

In particular, the claim is that, for $E$ and $W$ as above, there is an equivalence of categories
\begin{equation} \label{bitensor} \mathfrak{Cat}(E\otimes_{\mathfrak C} W,\A)\simeq \mathfrak{Hom}(\mathfrak C^{op},\mathfrak{Cat})(W,\mathfrak{Cat}(E,\A))
\end{equation}
for any category $\A$.  Further, the equivalence is pseudo-natural in $W$ and $\A$, making the evident functor $E\otimes_{\mathfrak C}-$ the left biadjoint of the ``hom'' functor $\mathfrak{Cat}(E,-)$.  This will further justify the tensor notation and the name.  But further, the equivalence \ref{bitensor} is exactly the definition of the bicolimit of $E$ weighted by $W$ as in Definition \ref{bicolim defn}.  Thus, the tensor construction will give a concrete computation of this weighted bicolimit.

There are at least two ways to go about a proof.  On the one hand, it could proceed in the manner of the proof of Theorem \ref{MAIN THEOREM}, showing explicitly the functorial assignments in each direction.  The proof is then a matter of showing well-definition and that the necessary coherence conditions obtain.  On the other hand, the following development sidesteps some of that complication.

\begin{lemma}[Technical Equivalence]  Let $\mathfrak C$ denote a small 2-category.  For each small category $\A$ and any pseudo-functors $E\colon \mathfrak C\to\mathfrak{Cat}$ and $W\colon \mathfrak C^{op}\to\mathfrak{Cat}$, there is an equivalence of categories
\begin{equation} \label{tech equiv} \Hom (\mathfrak C^{op}\times \mathfrak C,\cat)(\mathfrak C(-,-),\cat (E\times W,\A))\simeq \Hom (\mathfrak C^{op},\cat )(W,\cat(E,\A)).
\end{equation}
Additionally, this equivalence is pseudo-natural in $W$ and in $\A$.
\end{lemma}
\begin{proof}  The proof of this result is a straightforward but tedious exercise in verifying coherence conditions.  For example, given pseudo-natural 
\[ \alpha\colon  \mathfrak C(-,-) \Rightarrow \cat (E\times W,\A)
\]
define $\Phi(\alpha)$ on the right side of \ref{tech equiv} to have value in $\A$
\[ \Phi(\alpha)_C(Y)(X):= \alpha_{C,C}(1_C)(X,Y)
\] 
where $C\in \mathfrak C_0$ and $(X,Y)\in EC\times WC$.  There is an evident assignment on arrows of $EC$ making $\alpha_{C,C}(1_C)(-,Y)$ a functor $EC\to \A$ for each $Y\in WC$.  Over all such $Y$, the $\Phi(\alpha)_C(Y)$ are the object assignments of a functor $WC\to\cat(EC,\A)$.  These functors $\Phi(\alpha)_C$ are the components of a pseudo-natural transformation $W\to\cat(E,\A)$ with coherence isomorphisms induced from those of $\alpha$.  The assignment for $\Phi$ on modifications is easier and makes $\Phi$ a functor.

On the other hand, given a pseudo-natural $\theta\colon W\to\cat(E,\A)$ on the right side of \ref{tech equiv}, a transformation $\Psi(\theta)$ can be given with value in $\A$
\[ \Psi(\theta)_{C,D}(f)(X,Y):= \theta_D(Y)(f_!X)
\]
for $C,D\in \mathfrak C_0$ and $f\colon C\to D$ any arrow and objects $X\in EC$ and $Y\in WD$.  Given an arrow $(u,v)\colon (X,Y)\to (Z,W)$ in $EC\times WD$, the arrow assignment $\theta_D(Y)(f_!X)\to \theta_D(W)(f_!Z)$ in $\A$ is that given by either way round the naturality square
$$\begin{tikzpicture}
\node(1){$\theta_D(Y)(f_!X)$};
\node(2)[node distance=2in, right of=1]{$\theta_D(Y)(f_!Z)$};
\node(3)[node distance=.8in, below of=1]{$\theta_D(W)(f_!X)$};
\node(4)[node distance=.8in, below of=2]{$\theta_D(W)(f_!Z).$};
\node(5)[node distance=.5in, right of=1]{$$};
\node(6)[node distance=.4in, below of=5]{$$};
\draw[->](1) to node [above]{$\alpha_D(Y)(f_!u)$}(2);
\draw[->](1) to node [left]{$\alpha_D(v)_{f_!X}$}(3);
\draw[->](2) to node [right]{$\alpha_D(v)_{f_!Z}$}(4);
\draw[->](3) to node [below]{$\alpha_D(W)(f_!u)$}(4);
\end{tikzpicture}$$
This assignment makes each $\Psi(\theta)_{C,D}(f)$ into a functor $EC\times WD\to\A$, each $\Psi(\theta)_{C,D}$ into a functor, and, with appropriate coherence isomorphisms coming from $\theta$, each $\Psi(\theta)$ into a pseudo-natural transformation $\mathfrak C(-,-) \Rightarrow \cat (E\times W,\A)$.  The arrow assignment for $\Psi$ is straightforward.  

From the definitions it follows that $\Phi\Psi=1$ holds strictly and that $\Psi\Phi\cong 1$ making $\Phi$ and $\Psi$ appropriately pseudo-inverse.  The equivalences are indeed appropriately pseudo-natural in $W$ and in fact strictly 2-natural in $\A$.  \end{proof}

\begin{remark}
Notice that it is in the second half of the proof of Lemma \ref{tech equiv} that the constructions go ``off the diagonal" in the sense that arbitrary arrows $f\colon C\to D$ must be accommodated in the construction since the hom-bifunctor is the source of the pseudo-natural transformation.  This explains why weighting by the hom-bifunctor computes the bicolimit as in the following theorem.
\end{remark}

\begin{theo}  The bicolimit of a pseudo-functor $E\colon \mathfrak C\to\cat$ weighted by a pseudo-functor $W\colon \mathfrak C^{op}\to\cat$ is given by 
\[ E\otimes_{\mathfrak C}W:= (E\times W)\star \mathfrak C(-,-) \simeq E\star_{bi}W
\]
that is, by the pseudo-colimit of the product bifunctor $E\times W\colon \mathfrak C\times \mathfrak C^{op}\to\cat$ weighted by the usual hom-bifunctor $\mathfrak C(-,-)\colon \mathfrak C^{op}\times\mathfrak C\to\cat$.
\end{theo}
\begin{proof}  The main result, Theorem \ref{MAIN THEOREM}, yields a pseudo-natural isomorphism 
\[ \cat(E\otimes_{\mathfrak C}W,\A)\cong \Hom (\mathfrak C^{op}\times \mathfrak C,\cat)(\mathfrak C(-,-),\cat (E\times W,\A))
\]
The previous technical result, Lemma \ref{tech equiv}, then finishes the proof by definition of the bicolimit $E\star_{bi}W$ in Definition \ref{bicolim defn}.  \end{proof}

\begin{cor}  The tensor product $E\otimes_{\mathfrak C} W$ is the object assignment of a left biadjoint to the hom 2-functor $\mathfrak{Cat}(E,-)$.
\end{cor}
\begin{proof}  This is just a matter of showing that the equivalence given by Propositions 3.8 and 3.8 is pseudo-natural.  In fact naturality in $\mathscr A$ is strict, as can easily be computed from the definitions.  \end{proof}

\subsection{Comparing Pseudo- and Bi-colimits}

\label{section comparing}

Let $E\colon \mathfrak C\to\cat$ and $W\colon \mathfrak C^{op}\to\cat$ denote pseudo-functors on a 2-category $\mathfrak C$.  The goal is to construct a comparison functor 
\[ G\colon E\otimes_{\mathfrak C}W \to E\star W
\]
from the bicolimit to the pseudo-colimit.  By construction, this can be induced from the underlying so-called ``diagonal categories"
\begin{equation}
\label{underlying comparison functor} G\colon \Delta(E\times W,\mathfrak C(-,-)) \to \Delta(E,W)
\end{equation}
and then by showing that the assignments are functorial, well-defined on connected components, and finally invert the appropriate ``cartesian morphisms."  First give the assignments for \ref{underlying comparison functor}.  On objects such as in the description of Proposition \ref{descr bicolim}, take
\begin{equation} \label{comparison obj assign} G(C,D,X,Y,f) := (C,X,f^*Y).
\end{equation}
Now, take an arrow
\[ (h,k,m,n,\alpha)\colon (C,D, X,Y, f) \to (A,B,U,V,g)
\]
as in the description of Proposition \ref{descr bicolim}.  Notice that there is thus an arrow $n\colon k^*Y\to V$ of $WD$.  Hitting this with $g^*$ and then $h^*$ yields the composite
$$\begin{tikzpicture}
\node(1){$f^*Y$};
\node(2)[node distance=1in, right of=1]{$(kgh)^*Y$};
\node(3)[node distance=1in, right of=2]{$h^*g^*k^*Y$};
\node(4)[node distance=1.2in, right of=3]{$h^*g^*V$};
\draw[->](1) to node [above]{$(\alpha^*)_Y$}(2);
\draw[->](2) to node [above]{$\cong$}(3);
\draw[->](3) to node [above]{$h^*g^*n$}(4);
\end{tikzpicture}$$
in $WA$.  Notice that the middle arrow above is a composite of various coherence isomorphisms coming with the pseudo-functor $W$.  In any event, denote this arrow for now by $t$.  Take the image of $(h,k,m,n,\alpha)$ under $G$ to be
\begin{equation} \label{comparison arrow assign}
G(h,k,m,n,\alpha): = (h,m, t).
\end{equation}
Finally, give the 2-cell assignment.  Start with a 2-cell $(\gamma,\theta)$ as in Proposition \ref{descr bicolim}.  The assignment is then
\begin{equation} \label{comparison cell assign}
G(\gamma,\theta):= \gamma.
\end{equation}

\begin{lemma} \label{underlying comparison functor lemma} The assignments \ref{comparison obj assign}, \ref{comparison arrow assign} and \ref{comparison cell assign} above induce a functor of 2-categories as in \ref{underlying comparison functor}.
\end{lemma}
\begin{proof}  There is nothing to prove in the $\mathfrak C$- and $E$-arguments of $\Delta(E,W)$.  For the $\mathfrak C$-argument follows by definition; and the $E$-argument is the usual one for the 2-category of elements of a pseudo-functor.  However, there is work to be done in the $W$-argument.  Functoriality at the levels of morphisms and 2-cells are routine arguments using coherence conditions and naturality assumptions.  It needs to be seen, however, that the 2-cell assignment \ref{comparison cell assign} is well-defined in the first place.  That is, a certain triangle as described in the construction of $\Delta(E,W)$ must commute.  The triangle in question appears as the outside in the following diagram:
$$\begin{tikzpicture}
\node(1){$f^*Y$};
\node(2)[node distance=1in, right of=1]{$$};
\node(3)[node distance=1in, right of=2]{$h^*g^*n^*Y$};
\node(4)[node distance=.8in, above of=3]{$(kgh)^*Y\cong h^*g^*k^*Y$};
\node(5)[node distance=.8in, below of=3]{$(ngm)^*Y\cong m^*g^*n^*Y$};
\node(6)[node distance=2in, right of=4]{$h^*g^*V$};
\node(7)[node distance=2in, right of=5]{$n^*g^*V.$};
\draw[->](1) to node [above]{$(\alpha^*)_Y$}(4);
\draw[->](1) to node [below]{$(\beta^*)_Y$}(5);
\draw[->](4) to node [right]{$h^*g^*(\theta^*)_Y$}(3);
\draw[->](3) to node [right]{$(\gamma^*)_{h^*n^*Y}$}(5);
\draw[->](4) to node [above]{$h^*g^*v$}(6);
\draw[->](5) to node [below]{$m^*g^*y$}(7);
\draw[->](6) to node [right]{$(\gamma^*)_{g^*V}$}(7);
\draw[->](3) to node [below]{$h^*g^*y$}(6);
\end{tikzpicture}$$
The three subdiagrams each commute, making the whole figure commutative.  The leftmost portion commutes by the 2-cell compatibility from Proposition \ref{descr bicolim}; the top right triangle commutes by the commutative triangle condition of Proposition \ref{descr bicolim}; and finally, the bottom right square commutes by the naturality of $\gamma^*$.  Thus, the 2-cell assignment is well-defined.  \end{proof}

\begin{lemma}  The functor of Lemma \ref{underlying comparison functor lemma} is well-defined on connected components and on the category of fractions forming $E\otimes_{\mathfrak C}W$.
\end{lemma}
\begin{proof}  Well-definition on connected components follows since the 2-cell assignment \ref{comparison cell assign} is basically just projection.  On the other hand, given a cartesian morphism $(h,k,m,n,\alpha)$, both $m$ and $n$ are invertible and $\alpha$ is an invertible 2-cell.  Thus, the image arrow by \ref{comparison arrow assign} is cartesian as well, hence invertible when passing to the category of fractions forming $E\star W$.  \end{proof}

\begin{theo} \label{comparison functor theorem} For any pseudo-functors $E$ and $W$ as in the discussion above, there is a well-defined, surjective-on-objects functor  
\[ G\colon E\otimes_{\mathfrak C}W \to E\star W
\]
from the bicolimit of $E$ weighted by $W$ to the pseudo-colimit.
\end{theo}
\begin{proof}  Insofar as it suffices to use the computations achieved in the main theorems of this paper, this result is just a formal restatement of the previous discussion and lemmas.  Surjective-on-objects is immediate by the definition of the object assignment in \ref{comparison obj assign} since the identifications in passing to connected components and then to categories of fractions are made on arrows and not on objects.  \end{proof}

\begin{remark} \label{no hope for comparison} There is no hope that in general there is any comparison functor in the other direction
\[ E\star W\to E\otimes_{\mathfrak C}W
\] 
induced from the underlying 2-categories.  While an hypothetical object assignment is evident, namely, that given by
\[ (C,X,Y)\mapsto (C,C,X,Y,1_C),
\]
because the indexing 2-category of the bicolimit is $\mathfrak C\times \mathfrak C^{op}$, there is no obvious way to get an arrow assignment with an appropriate mixed variance.  Finally, consider the following example, showing that in general pseudo- and bi-colimits are inequivalent.
\end{remark}

\begin{example}  Let $\mathfrak C$ denote the indexing 2-category generated by a single object $X$, a single non-identity idempotent endomorphism $x\colon X\to X$ and a single non-identity idempotent 2-cell $\xi\colon x\Rightarrow x$.  The pseudo-functors $E$ and $W$ are each given by the constant functor with value $\mathbf 1$.  The pseudo-limit is then a groupoid, whereas the bilimit is not.  For the bilimit is formed through a right calculus of fractions as in Remark \ref{calculus of fractions comment} and the image of the generator arrow $(1, 1, 1,1, \xi)$ in the bilimit is not invertible.  Hence the pseudo-colimit and bi-colimit cannot be equivalent.
\end{example}

\section{Prospectus}

\label{prospectus}

In this final section there appears some reflection and speculation upon the direction of future work proceeding from the computations and result appearing so far in the paper.

\subsection{Pseudo- and Bicolimits Indexed by Bicategories}

The main theorem of the paper has presented a computation of the pseudo-colimit $E\star W$ of a pseudo-functor $E\colon \mathfrak C\to\cat$ weighted by a pseudo-functor $W\colon \mathfrak C^{op}\to\cat$ using connected components and the category of fractions construction.  A technical consequence has been that the bicolimit of $E$ weighted by $W$ has been computed as a pseudo-colimit of a product bifunctor weighted by the canonical ``hom" functor given by $\mathfrak C$.  This has all been done assuming that $\mathfrak C$ is a strict 2-category.

A question then naturally arises as to whether, or to what extent, the present computations can be adapted to cases where the structure on the domain or codomain of $E$ and $W$ is somehow relaxed or generalized.  In one case, for example, it might be that the domain is actually a bicategory $\mathfrak B$ in the sense of \cite{Benabou}.  It will be left to subsequent work to discuss the possibility of adapting the construction in \ref{colimit construction} to compute a weighted pseudo-colimit in this case.  The question is to what the extent the coherence problems introduced by the relaxed structure of $\mathfrak B$ can be dealt with intrinsically by taking connected components in the formation of the pseudo-colimit.

\subsection{2-Sheaves and Stacks}

\label{sheaves stacks prospectus section}

Alternatively, it might be the case that the constructions of the paper will give computations of pseudo- and bicolimits of pseudo-functors $E\colon \mathfrak C\to\mathfrak K$ for more general $\mathfrak K$ than just the case $\mathfrak K = \cat$ treated so far.  In particular, the case of $\mathfrak K = \Hom(\X^{op},\cat)$ for a small 1-category $\X$ appeared in a previous version of this paper.  It was seen there that the colimit construction of this paper could be carried out ``pointwise" for each $X\in\X_0$ to give a computation of the pseudo-colimit in $\Hom(\X^{op},\cat)$ as a \emph{bona fide} pseudo-functor.  In boosting up $\X$ to a small 2-category $\mathfrak X$, however, the technical difficulty of this extended result increases considerably.  Thus, the treatment of the case of $\mathfrak K = \Hom(\mathfrak X^{op},\cat)$ will be left to a subsequent paper.

Stacks and 2-sheaves appear as 2-dimensional analogues of ordinary sheaves.  These 2-categories are developed as reflective sub-2-categories of the 2-categories of pseudo-functors $\Hom(\C^{op},\cat)$ or 2-functors $[\mathfrak C^{op},\cat]$ once the underyling 2-category $\mathfrak C$ has been equipped with some appropriate notion of ``coverage" such as a Grothendieck topology (on $|\mathfrak C|$, for example).  These developments appear, for example, in \cite{StreetSheafThy} and \cite{StreetFibsinBicats}.  Insofar as weighted pseudo- and bicolimits of reflective sub-2-categories are computed by taking the colimit in the super-2-category and applying the reflector, weighted pseudo- and bicolimits of stacks and 2-sheaves will be given using the constructions of this paper, adapted to the parameterized case of 2-presheaves discussed above.

\subsection{2-Filteredness}

\label{prospectus on filteredness}

The original motivation for the present computation of weighted pseudo-colimits $E\star W$ was to give an account of ``filteredness conditions" characterizing flat category-valued pseudo-functors on 2-categories generalizing the development for 1-categories.

The material for set-valued functors on 1-categories is well-established.  Recall that a set-valued functor $Q\colon \C\to\mathbf{Set}$ is on a small 1-category is flat if the tensor product extension
\begin{equation} \label{set-tensor} Q\otimes_{\C}-\colon [\C^{op},\mathbf{Set}] \longrightarrow \mathbf{Set}
\end{equation}
along the Yoneda embedding is left exact in the sense that it preserves, up to isomorphism, finite limits.  Flat set-valued functors are completely characterized in terms of ``filteredness conditions" on the associated category of elements.  Recall that a category $\X$ is filtered if
\begin{enumerate}
\item it has an object;
\item any two objects $X,Y\in \X_0$ fit into a span $X\leftarrow Z\to Y$;
\item any parallel arrows $f,g\colon X\rightrightarrows Y$ are equalized by an arrow $h\colon Z \to X$, in that $fh=gh$.
\end{enumerate}
(This terminology is consistent with the usage of \S VII.6 of \cite{MM} whereas other authors would use the term ``cofiltered" instead.)  The characterization is then the following.

\begin{theo} \label{flat iff filtered presheaves} A copresheaf $Q\colon \C\to\mathbf{Set}$ is flat if, and only if, either of the following equivalent conditions hold.
\begin{enumerate}
\item Its category of elements $\displaystyle \int_{\C}Q$ is filtered as above.
\item The copresheaf $Q\colon \C\to\mathbf{Set}$ is canonically a filtered colimit of representable functors.
\end{enumerate} 
\end{theo} 
\begin{proof} For the first, see Theorem VII.6.3 of \cite{MM} and its proof.  For the second, note that by Theorem III.7.1 of \cite{MacLane}, the functor $Q$ is always colimit over its category of elements. \end{proof}

Thus, with an analogue of \ref{set-tensor} for pseudo-functors $E\colon \mathfrak C\to\cat$ on a given 2-category $\mathfrak C$, it can be asked what would be the ``higher filteredness conditions" on the 2-category of elements of $E$ equivalent to the exactness of the tensor associated to $E$.  This question was first addressed and answered in \cite{DDS} where in general a notion of $\sigma$-bifilteredness was shown to be equivalent to the assumption that a certain left bi-Kan extension preserves all finite bilimits.

The overall goal of the present work, however, is to show how such a result can be achieved using the colimit computations of this paper via the calculus of fractions as in Remark \ref{calculus of fractions comment}.  By the formal similarities of the universal properties and the work of this paper, two candidates for the the analogue of \ref{set-tensor} are in fact the functors induced by the pseudo-colimit
\begin{equation}
E\otimes_{\mathfrak C}-:= E\star - \colon \Hom(\mathfrak C^{op},\cat)\to\cat
\end{equation}
and by the bicolimit
\begin{equation}
E\otimes_{\mathfrak C}:= E\star_{bi}- \colon \Hom(\mathfrak C^{op},\cat)\to\cat
\end{equation}
whose values on weights $W\colon \mathfrak C^{op}\to\cat$ have been shown to be computed using the presentations of Equation \ref{colimit construction} and Proposition \ref{descr bicolim}, respectively.  Future work will show how under the assumption of left-exactness, necessary 2-filteredness conditions on the 2-category of elements of $E$ can be extracted.  Further it will be seen how, under these filteredness assumptions, the colimit $E\star W$ arises using a right calculus of fractions and how this description of the tensor allows for a direct verification of finite limit-preservation.  Part of this work and an ``elementary generalization" appears in \cite{LambertThesis} where the calculus of fractions is described in the language of internal categories.

\section{Acknowledgments}

Much of the work of the paper was included in the author's thesis \cite{LambertThesis}.  That research was supported by the NSERC Discovery Grant of Dr. Dorette Pronk at Dalhousie University and by NSGS funding through Dalhousie University.  The author would like to thank Dr. Pronk for supervising the research and for numerous conversations about the much of the content of the paper.  The author would also like to thank Dr. Pieter Hofstra, Dr. Robert Par\'e and Dr. Peter Selinger for comments and suggestions on an earlier version of the work in the paper.

\bibliography{research}
\bibliographystyle{alpha}

\end{document}